\documentclass[11pt,a4paper]{article}
\usepackage{amsmath}
\usepackage{bbm}
\usepackage{amssymb}
\usepackage{amscd}
\usepackage{graphicx}
\usepackage{caption}
\usepackage{chngcntr}
\usepackage{textcomp}
\usepackage{float}
\usepackage{subfig}
\usepackage{multirow}
\usepackage{booktabs}
\usepackage{array}
\usepackage[colorlinks,citecolor=blue,urlcolor=blue]{hyperref}
\usepackage{color}
\usepackage{amsthm}
\usepackage{enumerate}
\usepackage[round]{natbib}
\defcitealias{EPA_ECATT}{EPA, 2025}  
\usepackage{bm}
\usepackage{algorithm}
\usepackage{algpseudocode}

\newtheorem{theorem}{Theorem}
\newtheorem{proposition}{Proposition}
\newtheorem{remark}{Remark}
\newtheorem{lemma}{Lemma}

\newtheorem{definition}{Definition}
\newtheorem{assumption}{Assumption}
\newtheorem{simulation}{Simulation}

\title{Robust linear regression under latent group heterogeneity\footnotetext{We thank to Professor Shige Peng from Shandong University for rounds of useful discussions.}}
\author{Xifeng Li\thanks{Corresponding author. School of Mathematics, Shandong University, PR China, lixfay@mail.sdu.edu.cn.}\quad and \quad Shuzhen Yang\thanks{Shandong University-Zhong Tai Securities Institute for Financial Studies, Shandong University, PR China, yangsz@sdu.edu.cn.}}

\date{}
\begin{document}
\maketitle
\begin{abstract}
Uncertainty is ubiquitous in real-world data, and the assumptions underlying classical linear regression models are often violated in practice. Inspired by the theory of sublinear expectation, we consider a linear regression model where the random intercept term has mean uncertainty and the error term has variance uncertainty. We develop a novel two-step approach, named Expectation-Maximization with Moving Block (EMMB), to estimate the model parameters. The proposed method requires no prior knowledge of group structures or change points. Theoretical properties of the estimators are established under mild regularity conditions. Simulation studies and a real-data application to $\mathrm{PM}_{2.5}$ concentration modeling in Beijing demonstrate the superiority of the proposed method: it captures substantial intercept heterogeneity overlooked by ordinary least squares and yields more accurate and interpretable estimates.\\

\end{abstract}

\section{Introduction}
In linear regression analysis, the ordinary least squares framework typically presupposes a constant intercept and homogeneous error variance. In practice, however, these assumptions are often violated, as both the intercept and error variance can exhibit substantial heterogeneity. For instance, in statistical modeling of $\mathrm{PM}_{2.5}$ concentrations, researchers commonly rely on meteorological factors and lagged pollution levels, yet unobserved factors such as emission intensity and regional transport often lead to substantial shifts in baseline pollution levels over time \citep{liang2015assessing}. Similar structures arise in many other fields, including regional economics, epidemiology, and finance \citep{eberhardt2011econometrics, austin2017intermediate, pouliot2016robust, hu2023arbitrage}. In these settings, intercept heterogeneity can be substantial relative to covariate effects, and ignoring it can distort coefficient estimates, reduce estimation efficiency, and impair prediction accuracy.

Most existing methods for handling heterogeneity focus on variations in regression coefficients or rely on known group structures. However, relatively little attention has been paid to settings where only the intercept and error variance vary across unknown groups. For example, standard fixed- and random-effects models require group membership to be known a priori, while classical change-point regression primarily focuses on shifts in regression coefficients. \cite{fan2024environment} propose environment invariant linear least squares to address endogeneity and heterogeneity among different environments, focusing on high-dimensional estimation and non-asymptotic theory. \cite{yan2024statistical} consider a four-regime segmented regression model for temporally dependent data with segmenting boundaries depending on multivariate covariates with nondiminishing boundary effects.

To address distribution uncertainty, \cite{peng2004filtration,peng2005nonlinear} introduced the theory of sublinear expectation. Specifically, the theory considers a family of probability measures $\{P_{\theta}\}_{\theta\in\Theta}$, with corresponding linear expectations $\{\mathbb{E}_{\theta}\}_{\theta\in\Theta}$, rather than a single probability measure $P_{\theta_0}$. Over the past two decades, the theory and methodology of sublinear expectation have matured into a well-established area of modern probability theory. Moreover, numerous valuable studies have contributed to this field within statistics; see, for example, \cite{lin2016k}, \cite{lin2017upper}, \cite{peng2020hypothesis}, \cite{jin2021optimal}, \cite{yang2023linear}, \cite{li2026solution}, among many others. For a comprehensive exposition of the theory of sublinear expectation, we refer readers to Peng's plenary lecture at the 2010 International Congress of Mathematicians \citep{PengICM2010} and his monograph \citep{peng2019nonlinear}.

Inspired by the sublinear expectation theory, we consider a linear regression model where the random intercept term $\nu$ has mean uncertainty and takes values in the interval $[\underline{\mu}, \overline{\mu}]$, and the error term $\varepsilon$ has variance uncertainty. To obtain consistent estimators of the regression coefficient vector $\bm{\beta}$ as well as the lower and upper bounds $\underline{\mu}$ and $\overline{\mu}$ of the random intercept, we propose the EMMB method. The proposed method first uses an EM algorithm to estimate $\bm{\beta}$ under an unknown group structure, treating the intercepts as missing data within small windows of size $n_0$. Then, a moving block procedure is used to estimate $\underline{\mu}$ and $\overline{\mu}$. Notably, the proposed EMMB method requires no prior knowledge of the number or locations of the underlying groups and is computationally simple.

The estimator $\hat{\bm{\beta}}_{\mathrm{EM}}$ obtained by the EM algorithm admits an explicit closed-form solution, thus obviating the need for iterative computation. Under mild regularity conditions, we establish the unbiasedness, consistency, and asymptotic normality of $\hat{\bm{\beta}}_{\mathrm{EM}}$, and further demonstrate the consistency of $\hat{\overline{\mu}}$ and $\hat{\underline{\mu}}$ obtained by the moving block procedure. Moreover, although our method does not require prior knowledge of change points, it remains fully applicable when such information is available, further enhancing its practical flexibility.

Comprehensive simulation studies demonstrate the superiority of the proposed EMMB method. It substantially outperforms OLS in settings with intercept heterogeneity and heteroscedasticity, while maintaining comparable performance when classical linear model assumptions hold. The practical utility of the EMMB method is further validated through an application to $\mathrm{PM}_{2.5}$ concentration modeling in Beijing. The method uncovers a wide intercept range (spanning nearly 150), which represents the unobserved day-specific baseline pollution level that OLS fails to capture. By accounting for such unobserved temporal heterogeneity, the EMMB method provides more accurate and interpretable estimates.

The remainder of the paper is organized as follows. Section \ref{sec:2} introduces the linear regression model under distribution uncertainty and develops the EMMB method for estimating the unknown parameters $\bm{\beta}, \underline{\mu}$ and $\overline{\mu}$. In Section \ref{sec:3}, we provide a theoretical analysis for the estimators obtained by the proposed method. Section \ref{sec:4} presents extensive simulation studies to evaluate the finite sample performance of the EMMB method. A real-data application to $\mathrm{PM}_{2.5}$ modeling is illustrated in Section \ref{sec:5}. Finally, Section \ref{sec:6} concludes the paper with some discussions. Technical proofs are relegated to the Appendix.

\section{Methodology}\label{sec:2}

\subsection{Model setup}\label{sec:2.1}

Let $\{F_{\theta}\}_{\theta\in\Theta}$ be a family of distributions, and let $\mathbb{E}_{\theta}$ denote the expectation under distribution $F_\theta$. For a random variable $\xi : \Omega \to \mathbb{R}$ such that $\mathbb{E}_{\theta}[|\xi|]<\infty$ for all $\theta \in \Theta$, its sublinear expectation is defined as
\begin{equation}\label{eq:sub}
\hat{\mathbb{E}}[\xi]:=\sup \limits_{\theta \in \Theta} \mathbb{E}_{\theta}[\xi].
\end{equation}
Here, we do not conduct our analysis directly in the sublinear expectation space introduced by \cite{peng2004filtration,peng2005nonlinear}, but merely draw on the above definition \eqref{eq:sub}. Other settings in this paper, such as the independent and identically distributed (i.i.d.) assumption, remain in the classical sense.

Throughout the paper, $\mathbb{E}_\theta[\cdot]$ denotes the expectation taken under a specific distribution $F_\theta$ in the family $\{F_{\theta}\}_{\theta\in\Theta}$. The symbols $P$ and $\mathbb{E}$ are used in the classical sense: for any random variable $\xi$,  $P(\xi \in A)$ and
$\mathbb{E}[\xi]$ are computed with respect to the actual distribution of $\xi$. When the distribution of $\xi$ is not uniquely specified by the model, these values may vary accordingly. The notation $\hat{\mathbb{E}}[\cdot]$ is reserved for the sublinear expectation defined in \eqref{eq:sub}.

We consider the following linear regression model
\begin{equation}\label{eq:lr}
y = \mathbf{x}^{\top}\bm{\beta} + \nu + \varepsilon,
\end{equation}
where $y \in \mathbb{R}$ is the response variable, $\mathbf{x}=(x_1,\ldots,x_p)^{\top} \in \mathbb{R}^p$ is the vector of predictors having a certain
distribution $F_{\mathbf{x}}(\cdot)$, and $\bm{\beta}=(\beta_1,\ldots,\beta_p)^{\top} \in \mathbb{R}^p$ is the vector of regression coefficients. It is assumed that $(\nu, \varepsilon)$ is independent of $\mathbf{x}$ and $\varepsilon$ is independent of $\nu$. Unlike traditional linear regression models, we assume that the random intercept term $\nu$ has mean uncertainty and takes values in the interval $[\underline{\mu}, \overline{\mu}]$ , and the error term $\varepsilon$ has variance uncertainty. In other words,
$$
\overline{\mu}=\hat{\mathbb{E}}[\nu],\qquad\underline{\mu}=-\hat{\mathbb{E}}[-\nu],\qquad\overline{\sigma}^{2}=\hat{\mathbb{E}}[\varepsilon^{2}],\qquad
\underline{\sigma}^{2}=-\hat{\mathbb{E}}[-\varepsilon^{2}],
$$
and $\hat{\mathbb{E}}[\varepsilon]=\hat{\mathbb{E}}[-\varepsilon]=0$.

The intervals $[\underline{\mu}, \overline{\mu}]$ and $[\underline{\sigma}^{2}, \overline{\sigma}^{2}]$ characterize the mean uncertainty and the variance uncertainty of $\nu$ and $\varepsilon$ respectively. A similar model setup appears in \cite{yang2023linear}. However, our assumptions about the data generation process and our estimation method (EMMB) are fundamentally different, as detailed in Section \ref{sec:2.2}.

To quantify the conditional mean under distribution uncertainty, we introduce the following conditional sublinear expectation
$$
\hat{\mathbb{E}}[y\mid \mathbf{x}]=\sup \limits_{\theta \in \Theta} \mathbb{E}_{\theta}[y\mid \mathbf{x}].
$$
This operator characterizes the maximum conditional mean of $y$ given $\mathbf{x}$ over all possible distributions in $\{F_{\theta}\}_{\theta\in\Theta}$. Based on this definition, we have the following proposition for the linear regression model with distribution uncertainty.

\begin{proposition}\label{prop:1}
For the linear regression model \eqref{eq:lr}, the conditional mean of $y$ given $\mathbf{x}$ satisfies
$$
\hat{\mathbb{E}}[y\mid \mathbf{x}]=\mathbf{x}^{\top}\bm{\beta} + \overline{\mu}
$$
and
$$
-\hat{\mathbb{E}}[-y\mid \mathbf{x}]=\mathbf{x}^{\top}\bm{\beta} + \underline{\mu}.
$$
\end{proposition}

\begin{remark}
 According to Proposition \ref{prop:1}, the conditional mean of $y$ given $\mathbf{x}$ lies in the interval $[\mathbf{x}^{\top}\bm{\beta} + \underline{\mu}, \mathbf{x}^{\top}\bm{\beta} + \overline{\mu}]$, where the upper and lower bounds are directly determined by the mean uncertainty interval $[\underline{\mu}, \overline{\mu}]$.
\end{remark}

\subsection{Our approach: EMMB}\label{sec:2.2}
In this paper, our goal is to obtain consistent estimates of $\bm{\beta}, \underline{\mu}$ and $\overline{\mu}$ using observations from linear regression model \eqref{eq:lr}. We consider a scenario where the data are generated from an unknown number $k$ of samples. Specifically, given $n$ independent observations $\{(\mathbf{x}_{i},y_{i})\}_{i=1}^{n}$, the $k$-sample setup is characterized by the next assumption.

\begin{assumption}\label{assum:ksample}
The independent observations $\{(\mathbf{x}_{i},y_{i})\}_{i=1}^{n}$ are composed of $k$ samples. For $1\leq j\leq k$, the sample size of the $j$th sample is $n_{j}$ and $\sum\limits_{j=1}^{k}n_{j}=n$. In the $j$th sample, $\nu_{i}$ follows the one-point distribution $\delta_{c_{j}}$, where $c_{j}\in [\underline{\mu},\overline{\mu}]$ with $\underline{\mu}=\min\limits_{1\leq j\leq k} c_{j}$ and $\overline{\mu}=\max\limits_{1\leq j\leq k}c_{j}$; $\varepsilon_{i}$ follows the normal distribution $\mathcal{N}(0, \sigma_{j}^{2})$ and $\sigma_{j}\in [\underline{\sigma},\overline{\sigma}]$ with $0 < \underline{\sigma} \leq \overline{\sigma} < \infty$. However, we do not know the number of the samples and the exact locations of the change points of the samples.
\end{assumption}

Borrowing the idea from the EM algorithm \citep{dempster1977maximum}, we treat $\{\nu_{i}\}_{i=1}^{n}$ as missing data and employ an alternating maximization procedure to estimate $\bm{\beta}$,  which we still refer to as the EM algorithm. With $\hat{\bm{\beta}}_{\mathrm{EM}}$ obtained via the EM algorithm, we then estimate $\underline{\mu}$ and $\overline{\mu}$ using the moving block method. We now propose the following two-step method, which is called Expectation-Maximization with Moving Block (EMMB).

\medskip

\noindent\textbf{Step 1. The EM algorithm for estimating }$\bm{\beta}$

Since the number of samples and the exact locations of change points are unknown, let $n_{0}<\min\limits_{1\leq j\leq k} n_{j}$ be a small window size and $T=n/n_{0}$. In practice, $n_{0}=10$ or 20 is recommended.

(a) The E-step: given $\bm{\beta}^{(s)}$ , find $a_{j}$ such that
$$
a_{j}=\arg\min_{a} \sum_{i=1+(j-1)n_{0}}^{jn_{0}} \left\{y_i - (\mathbf{x}_{i}^{\top}\bm{\beta}^{(s)}+a)\right\}^{2}.
$$
Differentiating the objective function with respect to $a$ and setting the derivative to zero yields
$$
\sum_{i=1+(j-1)n_{0}}^{jn_{0}} \left\{y_i - (\mathbf{x}_{i}^{\top}\bm{\beta}^{(s)}+a_{j})\right\}=0,
$$
which has the explicit solution
$$
a_{j}=\frac{1}{n_0}\sum_{i=1+(j-1)n_{0}}^{jn_{0}} (y_i - \mathbf{x}_{i}^{\top}\bm{\beta}^{(s)}).
$$
Then let $\nu_{1+(j-1)n_{0}}^{(s+1)}=\cdots=\nu_{jn_{0}}^{(s+1)}=a_{j}$, for $1\leq j\leq T$.

~\\

(b) The M-step: given $\bm{\nu}^{(s+1)}$, find $\bm{\beta}^{(s+1)}$ such that
$$
\bm{\beta}^{(s+1)}=\arg\min_{\bm{\beta}} \sum_{i=1}^n \left\{y_i - (\mathbf{x}_{i}^{\top}\bm{\beta}+\nu_{i}^{(s+1)})\right\}^{2},
$$
Differentiating the objective function with respect to $\bm{\beta}$ and setting the gradient to zero gives the normal equations
$$
\sum_{i=1}^n \{ y_i -(\mathbf{x}_{i}^{\top}\bm{\beta}^{(s+1)}+\nu_{i}^{(s+1)})\}\mathbf{x}_{i}=\mathbf{0}.
$$
Initialize $\bm{\beta}^{(0)}$, then repeat (a)-(b) until convergence  (e.g., until $\|\bm{\beta}^{(s+1)} - \bm{\beta}^{(s)}\| < \delta$ for a prescribed tolerance $\delta$) to obtain the final estimate $\hat{\bm{\beta}}_{\mathrm{EM}}$.

\medskip

\noindent\textbf{Step 2.  The moving block procedure for estimating $\underline{\mu}$ and $\overline{\mu}$}

The data $\{(\mathbf{x}_{i},y_{i})\}_{i=1}^{n}$ are scanned sequentially into $L=n-w+1$ overlapping moving blocks of a given block length $w$ with $w > n_0$. For $1\leq l \leq L$, denote the data in the $l$th block by $B_{l}=\{(\mathbf{x}_{i},y_{i})\}_{i=l}^{l+w-1}$, then given $\hat{\bm{\beta}}_{\mathrm{EM}}$ from step 1, find $a_{l}$ such that
$$
\sum_{i=l}^{l+w-1} \left\{y_i - (\mathbf{x}_{i}^{\top}\hat{\bm{\beta}}_{\mathrm{EM}}+a_{l})\right\}=0.
$$
Let
$$
\hat{\underline{\mu}}=\min_{1\leq l\leq L}a_{l}\qquad \text{ and }\qquad \hat{\overline{\mu}}=\max_{1\leq l\leq L}a_{l}.
$$

\begin{remark}
The choice of $w$ involves a bias-variance trade-off: a larger $w$ reduces the variance of each $a_l$ but increases the risk that a block spans multiple change points, which introduces bias; a smaller $w$ is more likely to lie within a homogeneous segment but yields noisier $a_l$. In practice, $w$ can be selected based on domain-specific knowledge, or by data-driven methods such as cross-validation.
\end{remark}

In contrast to the method of \cite{yang2023linear}, which estimates the regression coefficients by selecting a single block with minimum variance, our EMMB approach utilizes all observations and does not require any prior knowledge of group structures or change points. Finally, we summarize the complete EMMB approach in Algorithm~\ref{alg:EMMB}.

\begin{algorithm}
\caption{EMMB approach}
\label{alg:EMMB}
\begin{algorithmic}[1]
\Statex \textbf{Input:} data $\{(\mathbf{x}_{i},y_{i})\}_{i=1}^{n}$, small window size $n_{0}$, maximum number of iterations $max\_iter$, convergence tolerance $tol$, block length $w$ ($w > n_0$)

\State \textbf{Initialize:}  $\bm{\beta} \gets \bm{\beta}^{(0)}$ (e.g., the OLS estimate)
\State $T \gets  n / n_0$

\For{$iter = 1$ \textbf{to} $max\_iter$}

    \State $\bm{\beta}_{old} \gets \bm{\beta}$
    \For{$j = 1$ \textbf{to} $T$}
         \State Solve for $a_{j}$ satisfying: $\displaystyle\sum_{i=1+(j-1)n_{0}}^{jn_{0}} \left\{y_i - ( \mathbf{x}_{i}^{\top}\bm{\beta}+a_{j})\right\}=0$
         \State Let $\nu_{1+(j-1)n_{0}}=\cdots=\nu_{jn_{0}}=a_{j}$
    \EndFor
    \State Solve for $\bm{\beta}$ satisfying: $\displaystyle\sum_{i=1}^n \{ y_i - (\mathbf{x}_{i}^{\top}\bm{\beta}+\nu_{i}) \}\mathbf{x}_{i}=0$

    \State $\bm{\beta}_{diff} \gets \|\bm{\beta} - \bm{\beta}_{old}\|_2$
    \If{$\bm{\beta}_{diff} < tol$}
        \State \textbf{break}
    \EndIf
\EndFor
\State $\hat{\bm{\beta}}_{\mathrm{EM}} \gets \bm{\beta}$
\State $L \gets n-w+1$
\For{$l = 1$ \textbf{to} $L$}
      \State Solve for $a_{l}$ satisfying: $\displaystyle\sum_{i=l}^{l+w-1} \left\{y_i - (\mathbf{x}_{i}^{\top}\hat{\bm{\beta}}_{\mathrm{EM}}+a_{l})\right\}=0$
\EndFor
\State $\hat{\overline{\mu}} \gets \max_{1\leq l\leq L}a_{l}$
\State $\hat{\underline{\mu}} \gets \min_{1\leq l\leq L}a_{l}$

\Statex \textbf{Output:} $\hat{\bm{\beta}}_{\mathrm{EM}}$, $\hat{\underline{\mu}}$ and $\hat{\overline{\mu}}$
\end{algorithmic}
\end{algorithm}

\section{Theory}\label{sec:3}
\subsection{Theoretical properties of $\hat{\bm{\beta}}_{\mathrm{EM}}$}

In this subsection, we investigate the theoretical properties of the EM estimator $\hat{\bm{\beta}}_{\mathrm{EM}}$. We first introduce the following assumption, which guarantees that each window of size $n_0$ lies entirely within a single underlying group and thus ensures the validity of the EM algorithm.

\begin{assumption}\label{assum:n0}
For $1\leq j\leq k$, $n_{j}$ is divisible by the small window size $n_{0}$.
\end{assumption}

For the analysis that follows, we introduce additional notation. Let $T = n/n_0$ denote the number of complete, non-overlapping windows. Under Assumptions \ref{assum:ksample} and \ref{assum:n0}, each window of size $n_0$ is contained entirely within a single true group, so that $ n = T n_0$. For $j = 1,\ldots,T$, let $\tau_j^2$ denote the error variance within the $j$-th window. Then $\tau_j^2 = \sigma_{g(j)}^2$, where $g(j) \in \{1,\ldots,k\}$ indexes the true group to which window $j$ belongs.

Then, we establish the following lemma showing that $\hat{\bm{\beta}}_{\mathrm{EM}}$ admits an explicit closed-form solution.

\begin{lemma}\label{lem:explicit}
Under Assumptions \ref{assum:ksample} and \ref{assum:n0}, the EM estimator $\hat{\bm{\beta}}_{\mathrm{EM}}$ admits an explicit closed-form solution, given by

\[
\hat{\bm{\beta}}_{\mathrm{EM}} = \bm{\beta} + \Bigl[ \sum_{j=1}^{T} \sum_{i=1+(j-1)n_0}^{j n_0} (\mathbf{x}_i - \bar{\mathbf{x}}_j)(\mathbf{x}_i - \bar{\mathbf{x}}_j)^{\top} \Bigr]^{-1} \sum_{j=1}^{T} \sum_{i=1+(j-1)n_0}^{j n_0} (\mathbf{x}_i - \bar{\mathbf{x}}_j)\varepsilon_i,
\]
where $\bar{\mathbf{x}}_j = \frac{1}{n_0} \sum_{i=1+(j-1)n_0}^{j n_0} \mathbf{x}_i$ for $1\leq j\leq T$.
\end{lemma}

Next, we establish the unbiasedness of the EM estimator $\hat{\bm{\beta}}_{\mathrm{EM}}$. To this end, we impose the following independence assumption. In addition to ensuring unbiasedness, this condition also facilitates the subsequent proofs of consistency and asymptotic normality.

\begin{assumption}\label{assum:mindependent}
$\mathbf{x}, \nu$ and $\varepsilon$ are mutually independent.
\end{assumption}

\begin{theorem}[Unbiasedness of $\hat{\bm{\beta}}_{\mathrm{EM}}$]\label{thm:unbiasednessLEM}
Under Assumptions \ref{assum:ksample}--\ref{assum:mindependent}, the EM estimator $\hat{\bm{\beta}}_{\mathrm{EM}}$ is unbiased for $\bm{\beta}$, i.e., $\mathbb{E}[\hat{\bm{\beta}}_{\mathrm{EM}}] = \bm{\beta}$.
\end{theorem}

Having established the unbiasedness of $\hat{\bm{\beta}}_{\mathrm{EM}}$, we next investigate its large-sample properties. Specifically, we require the following standard regularity conditions to prove the consistency and asymptotic normality of the EM estimator $\hat{\bm{\beta}}_{\mathrm{EM}}$.

\begin{assumption}\label{assum:x}
$\{\mathbf{x}_i\}_{i=1}^{n}$ are i.i.d. with \(\mathbb{E}[\|\mathbf{x}\|^3] < \infty\), and the covariance matrix \(\bm{\Sigma}_{\mathbf{xx}} = \mathrm{Var}(\mathbf{x}) = \mathbb{E}[(\mathbf{x} - \mathbb{E}[\mathbf{x}])(\mathbf{x} - \mathbb{E}[\mathbf{x}])^{\top}]\) is positive definite.
\end{assumption}

\begin{assumption}\label{assum:fixedn0}
As $n \to \infty$, $n_0$ is fixed and $T \to \infty$.
\end{assumption}

\begin{theorem}[Consistency of $\hat{\bm{\beta}}_{\mathrm{EM}}$]\label{thm:consistencyLEM fixed}
Under Assumptions \ref{assum:ksample}--\ref{assum:fixedn0}, the EM estimator $\hat{\bm{\beta}}_{\mathrm{EM}}$ is consistent for \(\bm{\beta}\), i.e.,
\[
\hat{\bm{\beta}}_{\mathrm{EM}} \xrightarrow{P} \bm{\beta} \quad \text{ as } n \to \infty.
\]
\end{theorem}

\begin{theorem}[Asymptotic Normality of $\hat{\bm{\beta}}_{\mathrm{EM}}$]\label{thm:ANLEM fixed}
Under Assumptions \ref{assum:ksample}--\ref{assum:fixedn0}, the EM estimator $\hat{\bm{\beta}}_{\mathrm{EM}}$ satisfies
\[
\sqrt{n}\bigl(\hat{\bm{\beta}}_{\mathrm{EM}} - \bm{\beta}\bigr) \xrightarrow{d} \mathcal{N}\bigl(\mathbf{0}, \, \frac{n_{0}}{n_0 - 1} \tilde{\sigma}_{T}^2\bm{\Sigma}_{\mathbf{xx}}^{-1}\bigr) \quad \text{ as } n \to \infty,
\]
where $\tilde{\sigma}_{T}^2=\frac{1}{T}\sum_{j=1}^{T} \tau_{j}^2\in[\underline{\sigma}^2,\overline{\sigma}^2]$ and $\tau_j^2$ is the error variance within the $j$-th window.
\end{theorem}

Notably, the asymptotic variance of the EM estimator $\hat{\bm{\beta}}_{\mathrm{EM}}$ differs from that of the OLS estimator under a constant intercept only by a multiplicative factor $n_{0}/(n_{0} - 1)$. Since $n_0$ is fixed at a moderate value (e.g., $10$ or $20$) in practice, this factor is close to 1 and has little effect on inference. This indicates that the EMMB method effectively eliminates the influence of intercept heterogeneity on the estimation of $\bm{\beta}$, achieving an efficiency comparable to that of OLS in the absence of such heterogeneity.

Furthermore, since the true subsample sizes $n_j$ are typically much larger than $n_0$, even if $n_j$ is not exactly divisible by $n_0$, at most $k-1$ windows may contain observations from two adjacent subsamples. The number of such boundary windows is asymptotically negligible, and their impact on $\hat{\bm{\beta}}_{\mathrm{EM}}$ vanishes as the total sample size $n$ grows. Therefore, the theoretical properties of $\hat{\bm{\beta}}_{\mathrm{EM}}$ remain practically meaningful and robust to minor violations of Assumption \ref{assum:n0}.

\subsection{Theoretical properties of $\hat{\underline{\mu}}$ and $\hat{\overline{\mu}}$}

To establish a firm theoretical foundation for $\hat{\underline{\mu}}$ and $\hat{\overline{\mu}}$ obtained by the proposed EMMB method, we first define the population-level targets that our procedure naturally estimates.

\begin{definition}[$w$-robust upper and lower bounds]\label{def:wbounds}
The observed data $\{(\mathbf{x}_{i},y_{i})\}_{i=1}^{n}$ are partitioned sequentially into $L=n-w+1$ overlapping moving blocks of length $w$. For the $l$th block $B_{l}=\{(\mathbf{x}_{i},y_{i})\}_{i=l}^{l+w-1}$, define its expected average intercept as
\[
\mu(B_{l}) = \frac{1}{w} \sum_{i=l}^{l+w-1} \mathbb{E}[\nu_i],
\]
where $\nu_i = c_j$ if the $i$th observation belongs to the $j$th homogeneous sample(as described in Assumption \ref{assum:ksample}). The $w$-robust upper bound $\overline{\mu}_w$ and $w$-robust lower bound $\underline{\mu}_w$ are then defined respectively by
\[
\overline{\mu}_w = \max_{1\leq l\leq L} \mu(B_{l}), \qquad
\underline{\mu}_w = \min_{1\leq l\leq L} \mu(B_{l}).
\]
\end{definition}

\begin{remark}
The quantities $\overline{\mu}_w$ and $\underline{\mu}_w$ represent the maximum and minimum expected average intercept across all sliding windows of length $w$. If a homogeneous subsample with the global extremum $\overline{\mu}$ (or $\underline{\mu}$) has length at least $w$, then $\overline{\mu}_w$ (or $\underline{\mu}_w$) coincides with the global maximum $\overline{\mu}$ (or global minimum $\underline{\mu}$). Thus, estimating $\overline{\mu}_w$ and $\underline{\mu}_w$ provides a conservative yet stable measure of risk over a horizon of length $w$, which is often more relevant for practical decision-making than the historical global extremum.
\end{remark}

To establish the consistency of the EMMB estimators $\hat{\overline{\mu}}$ and $\hat{\underline{\mu}}$, we impose the following standard regularity condition on the block length $w$, which ensures the convergence of the block averages and the validity of the extreme value analysis.

\begin{assumption}\label{assum:w}
As the total sample size $n \to \infty $, the block length $w$ satisfies
\[
w \to \infty, \quad w = o(n), \quad \text{and} \quad \frac{w}{\log n} \to \infty.
\]
\end{assumption}

\begin{theorem}[Consistency of $\hat{\overline{\mu}}$ and $\hat{\underline{\mu}}$]\label{thm:consistency_bounds}
Under Assumptions \ref{assum:ksample}-\ref{assum:w}, the EMMB estimators $\hat{\overline{\mu}}$ and $\hat{\underline{\mu}}$ are consistent for the $w$-robust upper bound $\overline{\mu}_w$ and $w$-robust lower bound $\underline{\mu}_w$, respectively, i.e.,
\[
\hat{\overline{\mu}} \xrightarrow{P} \overline{\mu}_w \quad \text{and} \quad \hat{\underline{\mu}} \xrightarrow{P} \underline{\mu}_w \quad \text{ as } n \to \infty.
\]
\end{theorem}

Theorem~\ref{thm:consistency_bounds} guarantees that our procedure consistently estimates the $w$-robust bounds $\overline{\mu}_w$ and $\underline{\mu}_w$. In practice, if the block length $w$ is chosen to reflect a relevant planning horizon (e.g., one quarter in financial risk assessment), then $\hat{\overline{\mu}}$ and $\hat{\underline{\mu}}$ provide statistically reliable estimates of the worst- and best-case average outcomes over any consecutive $w$-period within the observed data. This makes the estimator particularly suitable for conservative planning and stress testing, where guarding against sustained adverse periods is more critical than reacting to isolated extreme points.

\subsection{Extension to settings with known group structure} \label{subsec:known_groups}

While our EMMB approach is developed for the case of unknown change points, it readily adapts to situations where the group structure (i.e., the partition of the data into homogeneous subsamples) is known. In such a setting, one can simply take each known group as a single block and apply the EM algorithm to these $k$ blocks to obtain $\hat{\bm{\beta}}$  and the group-specific intercept estimates $\{a_j\}_{j=1}^{k}$. Natural estimators for the lower and upper bounds of the intercepts are then

$$
\hat{\underline{\mu}}=\min_{1\leq j\leq k}a_{j}\qquad \text{ and }\qquad \hat{\overline{\mu}}=\max_{1\leq j\leq k}a_{j}.
$$

We can similarly prove the theoretical properties of $\hat{\bm{\beta}}$, $\hat{\underline{\mu}}$, and $\hat{\overline{\mu}}$. Specifically, $\hat{\bm{\beta}}$ is still unbiased, and as the sample size of each group tends to infinity (i.e., $n_j \to \infty$), $\hat{\bm{\beta}}$ remains consistent and asymptotically normal, while $\hat{\underline{\mu}}$ and $\hat{\overline{\mu}}$ are consistent for the corresponding population bounds. Furthermore, when the group structure is known, the M-step of the EM algorithm can be weighted by the inverse of the within-group variances to obtain a more efficient minimum-variance-type estimator for $\bm{\beta}$.

This extension highlights the flexibility of the proposed method: it provides a unified estimation strategy that works seamlessly whether the heterogeneity of the data is latent (with unknown change points) or manifest (with known groups).

\section{Simulation studies}\label{sec:4}
In this section, we conduct several simulation experiments to systematically evaluate and compare the finite sample performance of the proposed EMMB method with the ordinary least squares (OLS). To facilitate clear comparisons, we consider the bias and mean squared error (MSE) of the regression coefficient estimates for both methods. In addition, we present the mean of the lower and upper intercept estimators $\hat{\underline{\mu}}$ and $\hat{\overline{\mu}}$ obtained by the EMMB method, as well as the mean of the overall intercept estimator $\hat{\beta_{0}}$ obtained by OLS. All simulation results are based on 1,000 repetitions.

\begin{simulation}\label{sim:1}
We consider the following linear regression model
$$
y = \mathbf{x}^{\top}\bm{\beta} + \nu + \varepsilon,
$$
where $\bm{\beta}=(1.5,-0.6)^{\top}$ and the components of $\mathbf{x}$ are independent and identically distributed as $\mathcal{N}(0,1)$. $\nu$ is a random intercept and $\varepsilon$ has variance uncertainty. The data are composed of k samples with equal size $m$ and thus $n=km$. Specifically, $\{\nu_{i}\}_{i=1}^{n}$ and $\{\varepsilon_{i}\}_{i=1}^{n}$ are generated as follows:
\begin{itemize}
  \item let $c_{1}=-10, c_{k}=30$ and $c_{j}$ is generated from the uniform distribution $U(-10,30)$ for $2\leq j\leq k-1$;
  \item $\sigma_{j}$ is generated from the uniform distribution $U(1,2)$ for $1\leq j\leq k$;
  \item let $\nu_{i}=c_{j}$ and $\varepsilon_{i}$ is generated from the normal distribution $\mathcal{N}(0, \sigma_{j}^{2})$  for $ 1+(j-1)m\leq i\leq jm, 1\leq j\leq k$.
\end{itemize}
This design ensures that the intercept range $[-10,30]$ is exactly realized, with the endpoints $-10$ and $30$ explicitly attained. The numerical
results associated with this simulation are reported in Table \ref{table:1}.
\end{simulation}

\begin{table}[tbp]
\centering
\caption{Simulation Results over 1,000 repetitions for Simulation \ref{sim:1} with $m=100$ and varying $k$}
\label{table:1}
\begin{tabular}{l cc cc c}
\toprule
\multirow{2}{*}{Method} & \multicolumn{2}{c}{$\hat{\beta}_{1}$} & \multicolumn{2}{c}{$\hat{\beta}_{2}$} & $(\hat{\underline{\mu}},\hat{\overline{\mu}})$ \\
\cmidrule(lr){2-3} \cmidrule(lr){4-5} \cmidrule(lr){6-6}
 & Bias & MSE & Bias & MSE & Mean \\
\midrule
\multicolumn{6}{l}{$(n,w,n_{0})=(200,80,10)$} \\
EMMB & $0.0027$ & $0.0130$  &  $-0.0039$ & $0.0126$  & $(-10.0740, 30.0810)$  \\
OLS & $-0.0067$ & $2.1061$  &$-0.0379$ & $2.1386$ & $10.0033$  \\[0.5em]
\multicolumn{6}{l}{$(n,w,n_{0})=(400,80,10)$} \\
EMMB & $0.0011$ & $0.0068$  &  $-0.0024$ & $0.0063$  & $(-10.0935, 30.0710)$  \\
OLS &  $0.0207$ & $0.6678$  &   $0.0104$ & $0.6140$ & $9.9935$  \\[0.5em]
\multicolumn{6}{l}{$(n,w,n_{0})=(1000,80,10)$} \\
EMMB & $0.0017$  & $0.0025$  &$0.0022$ & $0.0024$  & $(-10.0779, 30.0826)$ \\
OLS & $0.0079$ & $0.1840$  & $0.0170$ & $0.1832$ & $9.9643$\\
\bottomrule
\end{tabular}
\end{table}

Table \ref{table:1} reveals that OLS suffers from severely inflated MSE when the sample size is small ($n=200$), a problem that gradually improves with larger sample size. In contrast, the EMMB method maintains low MSE across all sample sizes and consistently outperforms OLS. Notably, both methods exhibit negligible bias, indicating that the superiority of EMMB stems from its substantially smaller variance. Furthermore, the EMMB method accurately recovers the true intercept range $[-10,30]$, while OLS provides only a point estimate that masks the underlying heterogeneity.

\begin{simulation}\label{sim:2}
We consider the following linear regression model
$$
y = \mathbf{x}^{\top}\bm{\beta} + \nu + \varepsilon,
$$
where $\bm{\beta}=(3, -1.5, 0.6)^{\top}$ and the covariates $\mathbf{x}$ follow a multivariate normal distribution with zero mean and covariance matrix
$$
\bm{\Sigma} = \begin{pmatrix}
4 & 2 & 0 \\
2 & 4 & 0 \\
0 & 0 & 9
\end{pmatrix}.
$$
$\nu$ is a random intercept and $\varepsilon$ has variance uncertainty. The data are composed of k samples with equal size $m$ and thus $n=km$. The generation of $\{\nu_i\}_{i=1}^n$ and $\{\varepsilon_i\}_{i=1}^n$ follows the same rules as Simulation \ref{table:1}. The numerical
results associated with this simulation are reported in Table \ref{table:2}.
\end{simulation}

\begin{table}[tbp]
\centering
\caption{Simulation Results over 1,000 repetitions for Simulation \ref{sim:2} with $m=100$ and varying $k$}
\label{table:2}
\begin{tabular}{l cc cc cc c}
\toprule
\multirow{2}{*}{Method} & \multicolumn{2}{c}{$\hat{\beta}_{1}$} & \multicolumn{2}{c}{$\hat{\beta}_{2}$}& \multicolumn{2}{c}{$\hat{\beta}_{3}$} & $(\hat{\underline{\mu}},\hat{\overline{\mu}})$ \\
\cmidrule(lr){2-3} \cmidrule(lr){4-5} \cmidrule(lr){6-7} \cmidrule(lr){8-8}
 & Bias & MSE & Bias & MSE& Bias & MSE & Mean \\
\midrule
\multicolumn{8}{l}{$(n,w,n_{0})=(200,80,10)$} \\
EMMB & $0.0010$ & $0.0043$  & $-0.0009$ & $0.0045$ &  $0.0006$ & $0.0016$  & $(-10.0705, 30.0791)$  \\
OLS &  $0.0059$ & $0.7694$  & $-0.0023$ & $0.6789$ & $0.0036$ & $0.2490$  & $10.0026$  \\[0.5em]
\multicolumn{8}{l}{$(n,w,n_{0})=(400,80,10)$} \\
EMMB & $0$ & $0.0022$  &  $0.0028$ & $0.0021$&   $-0.0003$ &  $0.0007$  & $(-10.0783, 30.0836)$  \\
OLS &   $0.0098$ & $0.2017$ &  $0.0051$ &  $0.2061$&  $-0.0055$ &  $0.0765$ & $9.9269$  \\[0.5em]
\multicolumn{8}{l}{$(n,w,n_{0})=(1000,80,10)$} \\
EMMB & $0.0028$  & $0.0009$  &$-0.0027$ & $0.0009$&$0.0003$ & $0.0003$  & $(-10.0819, 30.0883)$ \\
OLS &  $0.0065$ & $0.0610$  &  $-0.0077$ & $0.0624$& $0.0032$ & $0.0187$ & $10.2325$\\
\bottomrule
\end{tabular}
\end{table}

Simulation \ref{sim:2} introduces correlation among the covariates and allows them to have different variances, thereby increasing the signal-to-noise ratio relative to Simulation \ref{sim:1}. As shown in Table \ref{table:2}, the MSE of OLS decreases considerably, especially for small sample sizes, yet remains non-negligible. In contrast, the MSE of EMMB stays consistently low across all settings. The bias of both methods is small, indicating that the superior performance of EMMB is driven primarily by its much smaller variance. Furthermore, the EMMB method accurately recovers the true intercept range $[-10,30]$, while OLS provides only a point estimate that masks the underlying heterogeneity. These results confirm that the advantages of EMMB are not limited to low-signal scenarios but persist even when the signal is stronger and the covariates are more complex.

\begin{simulation}\label{sim:3}
We consider the following linear regression model
$$
y = \mathbf{x}^{\top}\bm{\beta} + \beta_{0} + \varepsilon,
$$
where $\bm{\beta}=(1.5,-0.6)^{\top}$ and the components of $\mathbf{x}$ are independent and identically distributed as $\mathcal{N}(0,1)$. The intercept is fixed at $\beta_{0}=-3$  and the error term $\varepsilon$ follows the normal distribution $\mathcal{N}(0,1)$. The numerical results associated with this simulation are reported in Table \ref{table:3}.
\end{simulation}

\begin{table}[tbp]
\centering
\caption{Simulation Results over 1,000 repetitions for Simulation \ref{sim:3} without uncertainty ($\beta_0=-3$)}
\label{table:3}
\begin{tabular}{l cc cc c}
\toprule
\multirow{2}{*}{Method} & \multicolumn{2}{c}{$\hat{\beta}_{1}$} & \multicolumn{2}{c}{$\hat{\beta}_{2}$} & $(\hat{\underline{\mu}},\hat{\overline{\mu}})$ \\
\cmidrule(lr){2-3} \cmidrule(lr){4-5} \cmidrule(lr){6-6}
 & Bias & MSE & Bias & MSE & Mean \\
\midrule
\multicolumn{6}{l}{$(n,w,n_{0})=(200,80,10)$} \\
EMMB & $0.0006$ & $0.0058$  & $-0.0010$ & $0.0058$  & $(-3.1494 -2.8610)$  \\
OLS & $0.0014$ & $0.0053$  & $-0.0003$ & $0.0052$ & $-3.0057$  \\[0.5em]
\multicolumn{6}{l}{$(n,w,n_{0})=(400,180,10)$} \\
EMMB & $-0.0005$ & $0.0029$  &   $-0.0015$ & $0.0027$  & $(-3.0927, -2.9150)$  \\
OLS &  $-0.0003$ & $0.0025$  &  $-0.0016$ &  $0.0024$ & $-3.0029$  \\[0.5em]
\multicolumn{6}{l}{$(n,w,n_{0})=(1000,480,10)$} \\
EMMB &  $-0.0002$  & $0.0012$  &$-0.0003$ & $0.0012$  & $(-3.0530, -2.9480)$ \\
OLS & $0.0001$ & $0.0010$  & $-0.0003$ & $0.0010$ & $-3.0004$\\
\bottomrule
\end{tabular}
\end{table}

From Table \ref{table:3}, we can see that when the intercept is constant and the error term has no variance uncertainty, the performance of the EMMB method is essentially on par with that of OLS. Both methods exhibit small bias and MSE across all sample sizes. Furthermore, the difference between $\hat{\underline{\mu}}$ and $\hat{\overline{\mu}}$ obtained by the EMMB method is very small, and as $w$ increases with the sample size $n$, these bounds gradually converge toward the true intercept value $\beta_{0}=-3$.

In summary, the EMMB method substantially outperforms OLS under Simulations \ref{sim:1} and \ref{sim:2}, where intercept heterogeneity and heteroscedasticity are present, while it performs comparably to OLS in Simulation \ref{sim:3}, where the classical linear model assumptions hold. Together, these three simulation experiments demonstrate the superiority of the proposed EMMB method: it achieves a substantial reduction in MSE when heterogeneity exists, without incurring any noticeable penalty when it does not.

\section{Real data analysis}\label{sec:5}
In this section, we apply the proposed EMMB method to model $\mathrm{PM}_{2.5}$ concentrations in Beijing, comparing its performance against the ordinary least squares (OLS) benchmark. The analysis is based on a widely used public dataset, the \textbf{Beijing PM2.5} dataset, from the UCI Machine Learning Repository. This dataset integrates hourly $\mathrm{PM}_{2.5}$ measurements recorded at the US Embassy in Beijing with contemporaneous meteorological data from the Beijing Capital International Airport.

The dataset spans from January 1, 2010 to December 31, 2014, covering 43,824 hourly observations. The original variables include $\mathrm{PM}_{2.5}$ concentration (pm2.5), temporal indicators (year, month, day, hour), and meteorological factors: dew point (DEWP), temperature (TEMP), pressure (PRES), combined wind direction (cbwd), cumulated wind speed (Iws), as well as cumulated hours of snow (Is) and cumulated hours of rain (Ir). cbwd includes four categories: northwest (NW), northeast (NE), southeast (SE), and calm and variable (cv). Iws is the cumulative wind speed from the start of each wind direction to the current hour; when the wind direction changes, Iws resets to the hourly wind speed of the new direction. Missing values, denoted as "NA" in the original data, appear only in $\mathrm{PM}_{2.5}$ measurements, while all meteorological variables are complete.

\subsection{Data preprocessing}
To construct a daily time series for our analysis, the hourly measurements were aggregated to the daily level. The aggregation rules are as follows:

\begin{itemize}
\item \text{pm2.5\_mean}: For each day, the daily mean $\mathrm{PM}_{2.5}$ concentration was calculated only if at least 18 hourly observations were available; otherwise, the daily value was marked as missing (NA). This threshold follows the completeness criteria of the U.S. Environmental Protection Agency (EPA) for ambient air quality monitoring \citepalias{EPA_ECATT}, which requires a minimum of 18 valid hourly observations (75\% of a full day) to compute a representative daily average for $\mathrm{PM}_{2.5}$. The daily average was computed as the arithmetic mean of the valid hourly measurements on that day.

\item \text{pm2.5\_lag4h}: The average $\mathrm{PM}_{2.5}$ concentration over the last four hours of the previous day, calculated using the same 75\% completeness criterion (i.e., at least three valid hourly observations).

\item \text{DEWP\_mean}, \text{TEMP\_mean}, \text{PRES\_mean}: Since DEWP, TEMP, and PRES have no missing values at the hourly level, daily means were computed directly from all 24 hourly observations.

\item \text{NE\_Iws\_inc}, \text{NW\_Iws\_inc}, \text{SE\_Iws\_inc}: To obtain physically interpretable daily measures, we first recovered the hourly wind speed by taking the difference in Iws between consecutive hours within the same wind direction segment. For each wind direction category (NE, NW, SE), the daily cumulative wind speed was then obtained by summing these hourly speeds over all hours belonging to that direction on that day. If a direction did not occur on a given day, its corresponding variable was set to zero. These variables reflect the total potential for pollutant transport or dispersion from each direction throughout the day.

\item \text{cv\_hours}: Since cv conditions---characterized by low hourly wind speeds (e.g., 0.45 m/s or 0.89 m/s)---tend to positively influence $\mathrm{PM}_{2.5}$ concentrations, using daily cumulative wind speed could misleadingly imply that higher wind speed leads to higher $\mathrm{PM}_{2.5}$ levels. To avoid this, we instead use the daily count of cv hours. This variable ranges from 0 to 24 and captures the duration of calm or variable conditions that favor pollutant accumulation.

\item \text{rain\_hours\_48h}: The total number of hours with rainfall over a 48-hour period, encompassing the preceding calendar day and the current day, was derived from the hourly cumulative rainfall indicator Ir. This variable accounts for the persistent cleansing effect of rainfall on air quality, which often extends into the day following a rainfall event.

\item \text{heating}: This dummy variable was constructed to account for the effect of residential heating, which is not directly available in the original dataset. The heating period in Beijing varies with temperature and can start earlier than November 15 or be extended after March 15. For example, it started on November 3 in 2012, and ended on March 22, 18, and 17 in 2010, 2012, and 2013, respectively. The indicator equals 1 for days within the effective heating period and 0 otherwise.
\end{itemize}
After aggregating all variables to the daily level, days with missing values (i.e., those where either \text{pm2.5\_mean} or \text{pm2.5\_lag4h} was not available) were excluded from the final dataset. Is was not used in this study due to its sparse occurrence and the focus on other meteorological drivers. The final daily dataset consists of 1,710 complete observations, with variables including \text{pm2.5\_mean}, \text{pm2.5\_lag4h}, \text{DEWP\_mean}, \text{TEMP\_mean}, \text{PRES\_mean}, \text{NE\_Iws\_inc}, \text{NW\_Iws\_inc}, \text{SE\_Iws\_inc}, \text{cv\_hours}, \text{rain\_hours\_48h} and \text{heating}.

\subsection{Results}
To account for the nonlinear cleansing effect of rainfall on $\mathrm{PM}_{2.5}$ concentrations, we applied a logarithmic transformation to the 48-hour cumulative rainfall hours, denoted as $\text{rain\_48h\_log1p} = \log(1+\text{rain\_hours\_48h})$. To facilitate interpretation of the interaction terms and reduce multicollinearity, all continuous meteorological variables (\text{DEWP\_mean}, \text{TEMP\_mean}, \text{PRES\_mean}, \text{cv\_hours}, \text{NE\_Iws\_inc}, \text{NW\_Iws\_inc}, \text{SE\_Iws\_inc}, and \text{rain\_48h\_log1p}) were centered prior to model fitting, while \text{pm2.5\_lag4h} and \text{heating} were left in their original scale. For notational simplicity, these variable names continue to denote their centered versions in the subsequent analysis.

According to \cite{liang2015assessing}, Beijing is hemmed in by Taihang Mountains to the west and Yanshan Mountains to the north, with few heavily polluting industries in the northern regions. In contrast, the southern and eastern parts of Beijing, situated on the North China Plain, host a high density of heavy industries that consume substantial amounts of coal and other fossil fuels. Under southerly wind conditions, the surrounding mountains obstruct air flow, leading to the accumulation of polluted air. To capture the seasonally varying impact of southeasterly winds, we introduced two interaction terms: $\text{SE\_Summer} = \text{SE\_Iws\_inc} \times \mathbf{1}_{\text{summer}}$ and $\text{SE\_Winter} = \text{SE\_Iws\_inc} \times \mathbf{1}_{\text{winter}}$.

We consider the following linear regression model for the daily $\mathrm{PM}_{2.5}$ concentration:
\begin{equation}\label{eq:model}
\begin{aligned}
\text{pm2.5\_mean} &= \beta_1 \text{pm2.5\_lag4h}+ \beta_2 \text{heating} +  \beta_3 \text{DEWP\_mean} \\
&+ \beta_4 \text{TEMP\_mean} + \beta_5 \text{PRES\_mean}+ \beta_6 \text{rain\_48h\_log1p}  \\
&+ \beta_7 \text{NE\_Iws\_inc} + \beta_8 \text{NW\_Iws\_inc}+ \beta_9 \text{SE\_Iws\_inc}\\
&+ \beta_{10} \text{cv\_hours} + \beta_{11} \text{SE\_Summer} + \beta_{12} \text{SE\_Winter}+\nu +  \varepsilon,
\end{aligned}
\end{equation}
where the random intercept term $\nu$ captures the day-specific baseline pollution level that may vary over time, and the error term $\varepsilon$ accounts for random fluctuations with potentially heterogeneous variance.

In practice, daily emission levels are not constant, which introduces uncertainty in the intercept $\nu$. Ordinary least squares (OLS) ignores this uncertainty by treating the intercept as a fixed constant $\beta_{0}$. In contrast, the proposed EMMB method assumes that the intercept is constant within each small window of size $n_{0}$, while allowing it to change across windows. The EMMB procedure first uses an EM algorithm to estimate the coefficient vector $\bm{\beta}$, and then applies a moving-block procedure to obtain more accurate estimates of the lower and upper bounds $\underline{\mu}$ and $\overline{\mu}$ of the intercept.

\begin{table}[t!]
\setlength{\abovecaptionskip}{0.1cm}
\setlength{\belowcaptionskip}{0.2cm}
\begin{center}
\caption{Parameter estimation and model fit for $\mathrm{PM}_{2.5}$ data with $n_{0}=10$ and $w=20$}
\label{table:model fit}
\begin{tabular}{l c c c c }
\toprule
 & \multicolumn{2}{c}{OLS} & \multicolumn{2}{c}{EMMB}\\
\cmidrule(lr){2-3} \cmidrule(lr){4-5}
 Parameters& Coef. & $p$-value & Coef. & $p$-value \\
\midrule
$\beta_1$ (\text{pm2.5\_lag4h})      & $0.4516$ & $<0.001$ & $0.3368$ &$<0.001$ \\
$\beta_2$ (\text{heating})          &$-2.9836$ & $0.496$ &$28.2025$ &$0.001$ \\
$\beta_3$ (\text{DEWP\_mean})       &$ 2.8768$ & $<0.001$ & $6.0453$ &$<0.001$  \\
$\beta_4$ (\text{TEMP\_mean})       &$-4.8913$ & $<0.001$ &$-5.1519$ &$<0.001$ \\
$\beta_5$ (\text{PRES\_mean})       &$-1.7873$ & $<0.001$ &$-1.6044$ &$<0.001$  \\
$\beta_6$ (\text{rain\_48h\_log1p}) &$-12.7810$& $<0.001$ &$-16.8265$&$<0.001$  \\
$\beta_7$ (\text{NE\_Iws\_inc})     &$-0.9888$ & $<0.001$ &$-0.7420$ &$<0.001$ \\
$\beta_8$ (\text{NW\_Iws\_inc})     &$-0.3383$ & $<0.001$ &$-0.2509$ &$<0.001$\\
$\beta_9$ (\text{SE\_Iws\_inc})     & $0.0028$ & $0.976$ &$-0.3234$ &$0.001$ \\
$\beta_{10}$ (\text{cv\_hours})     & $2.2654$ & $<0.001$ & $1.8224$ &$<0.001$\\
$\beta_{11}$ (\text{SE\_Summer})    & $0.2652$ & $0.042$ & $0.6072$ &$<0.001$ \\
$\beta_{12}$ (\text{SE\_Winter})    &$-0.1733$ & $0.253$ &$-0.5501$ &$0.001$ \\
$\overline{\mu}$                    & --     &--     &$126.2435$&--  \\
$\underline{\mu}$                   & --     &--     &$-26.0904$&--  \\
$\beta_0$ (Intercept)               & $49.0624$& --    & --     &--   \\
\midrule
$R^2$ & \multicolumn{2}{c}{$0.697$} & \multicolumn{2}{c}{$0.768$} \\
Adjusted $R^2$ & \multicolumn{2}{c}{$0.695$} & \multicolumn{2}{c}{$0.741$} \\
\bottomrule
\end{tabular}
\end{center}

\footnotesize
Note: $p$-values reported as $<0.001$ indicate that the exact $p$-value is less than $0.001$. The $R^2$ and adjusted $R^2$ for EMMB are calculated after including the estimated window-specific intercepts.
\end{table}

Based on 1,710 daily observations, we fitted model \eqref{eq:model} using both OLS and the proposed EMMB method, with results summarized in Table \ref{table:model fit}. By allowing the intercept to vary across windows of size $n_0=10$, the EMMB model substantially improves explanatory power: the $R^2$ increases from 0.697 to 0.768, a gain of 7.1 percentage points. The adjusted $R^2$ also rises from 0.695 to 0.741, confirming that this improvement is not due to overfitting. To estimate the baseline pollution level, we set the moving block length to $w = 20$ (roughly 20 calendar days, as missing data are sparse), which balances smoothing short-term fluctuations with capturing persistent baseline changes. The resulting intercept estimates range from $-26.09$ to $126.24$, a span of nearly 150. This wide range indicates strong temporal heterogeneity in the baseline pollution level that is entirely missed by the constant-intercept OLS model. Such heterogeneity may stem from day-to-day variations in emission intensity, unrecorded meteorological events, or policy interventions.

The heating indicator is insignificant in OLS ($p=0.496$) and even carries a negative sign, contradicting the expected positive effect of winter heating. In contrast, the EMMB method yields a positive and highly significant coefficient ($\beta=28.2025$, $p=0.001$), consistent with the well-documented impact of winter coal combustion. Dew point temperature (DEWP) is recognized as a key driver of $\mathrm{PM}_{2.5}$ pollution, as it promotes the trapping of pollutants under stagnant atmospheric conditions. However, the coefficient of \text{DEWP\_mean} in OLS is substantially underestimated ($\beta=2.8768$), which is less than half of that obtained by EMMB ($\beta=6.0453$). These discrepancies suggest that OLS suffers from coefficient estimation bias due to its fixed intercept assumption, while EMMB effectively isolates the true effects of these critical drivers.

EMMB uncovers state-dependent wind effects obscured by the restrictive constant-intercept framework of OLS. Across both methods, northwesterly (NW) and northeasterly (NE) winds consistently facilitate the removal of $\mathrm{PM}_{2.5}$. In contrast, SE\_Iws\_inc is negligible in OLS ($\beta=0.0028$, $p=0.976$) but emerges as a significant negative predictor in EMMB ($\beta=-0.3234$, $p=0.001$), suggesting that southeasterly winds exert a removal effect once baseline pollution levels are properly accounted for. The interaction terms SE\_Summer and SE\_Winter are both statistically significant in EMMB, with opposite signs: positive in summer ($\beta=0.6072$, $p<0.001$) and negative in winter ($\beta=-0.5501$, $p=0.001$). By comparison, SE\_Summer is only marginally significant ($\beta=0.2652$, $p=0.042$) in OLS, while SE\_Winter is not significant ($\beta=-0.1733$, $p=0.253$).

Estimates for the 4-hour lagged $\mathrm{PM}_{2.5}$ (\text{pm2.5\_lag4h}) and other meteorological variables (\text{TEMP\_mean}, \text{PRES\_mean}, \text{rain\_48h\_log1p}, and \text{cv\_hours}) remain stable across OLS and EMMB, confirming the consistency of core environmental mechanisms. The minor differences in magnitude reflect EMMB's ability to adjust for time-varying intercepts without distorting the underlying physical relationships.

Overall, these results demonstrate that EMMB provides more accurate and interpretable estimates by accounting for unobserved temporal heterogeneity, outperforming the restrictive OLS framework.

\subsection{Discussion of wind effects}
One might wonder: given that few heavily polluting industries lie north of Beijing while a high concentration of such industries is located to the southeast, why does the southeasterly wind appear to exert a stronger removal effect than the northwesterly wind in winter? Examination of hourly data reveals that $\mathrm{PM}_{2.5}$ concentrations reach their lowest levels during prolonged northwesterly winds, whereas extended southeasterly winds first increase and then decrease $\mathrm{PM}_{2.5}$ levels. Importantly, when southeasterly and northwesterly winds alternate, even if northwesterly winds prevail for most of a day, just a few hours of southeasterly wind can reintroduce pollution. Owing to the mountains northwest of Beijing, pollutants transported by southeasterly winds can accumulate; when the wind suddenly shifts to northwesterly, the initial northwesterly flow is not clean but instead carries air with elevated $\mathrm{PM}_{2.5}$ concentrations. Only after the northwesterly wind persists long enough to flush away these accumulated pollutants does it become effective cleansing winds. However, if northwesterly winds persist excessively, $\mathrm{PM}_{2.5}$ concentrations have already fallen to very low levels, diminishing their observable cleaning effect.

In winter, northwesterly winds predominate in Beijing while southeasterly winds are infrequent, causing the pollution-enhancing effect of southeasterly winds to be largely obscured by the dominant removal effect of northwesterly winds. In summer, the opposite pattern emerges: southeasterly winds prevail and northwesterly winds are rare, allowing the pollution-transporting effect of southeasterly winds to manifest clearly. Moreover, even though the coefficient for northwesterly wind is slightly smaller than that for southeasterly wind in spring and autumn, the overall removal effect of northwesterly wind remains substantially stronger. This is because NW\_Iws\_inc spans a much wider range (0 to 238.7) compared to SE\_Iws\_inc (0 to 113.13), meaning that northwesterly winds can exert a far greater cumulative removal effect when they do occur.

In summary, $\mathrm{PM}_{2.5}$ pollution in Beijing arises from both local emissions and regional transport. Southeasterly winds are a major conveyor of externally sourced pollutants, while prolonged northwesterly winds promote efficient $\mathrm{PM}_{2.5}$ removal. However, frequent alternation between the two wind directions can neutralize the removal effect of northwesterly winds and even temporarily transform them into a source of pollution.

\section{Conclusion}\label{sec:6}

In this paper, we consider a linear regression model with distribution uncertainty. Inspired by the theory of sublinear expectation, the model explicitly captures mean uncertainty in the random intercept term and variance uncertainty in the error term. We develop the Expectation-Maximization with Moving Block (EMMB) method to estimate the unknown parameters $\bm{\beta}, \underline{\mu}$ and $\overline{\mu}$, without requiring prior knowledge of group structures or change points. Under mild regularity conditions, we establish the unbiasedness, consistency, and asymptotic normality of $\hat{\bm{\beta}}_{\mathrm{EM}}$, and further prove the consistency of $\hat{\overline{\mu}}$ and $\hat{\underline{\mu}}$ obtained by the moving block procedure. The superiority of the EMMB method is illustrated through extensive simulation studies and a real-data application to PM$_{2.5}$ concentration modeling in Beijing. The real data analysis shows that the EMMB method effectively captures unobserved temporal heterogeneity in baseline pollution levels, yielding more accurate and interpretable estimates than OLS.

There are several promising directions for future research. First, since we assume that both the number and location of change points are unknown, our method currently addresses primarily intercept heterogeneity. This is because when the number and location of change points are unknown, the error variance may vary arbitrarily across small windows (e.g., every 10 observations), making reliable estimation of segment-specific variances infeasible. Consequently, it is not possible to incorporate variance weighting in the EM algorithm to yield a minimum-variance estimator. A natural extension is to develop novel methods that simultaneously accommodate both intercept heterogeneity and heteroscedasticity of the error term, even without prior information on change points. Second, owing to specific technical constraints, we have thus far only leveraged the core idea of sublinear expectation, rather than its full theoretical framework. A more ambitious direction involves formulating a linear regression model within the complete sublinear expectation paradigm---where the random intercept follows a maximal distribution, the error term follows a G-normal distribution, and the variables satisfy nonlinear independence. Such a framework would more faithfully capture the inherent complexities of real-world data, particularly the ubiquitous distributional uncertainty.

\appendix

\section{Proofs}

\begin{proof}[Proof of Proposition \ref{prop:1}]
By the linear regression model \eqref{eq:lr}, we have
\begin{align*}
\hat{\mathbb{E}}[y\mid \mathbf{x}]&=\hat{\mathbb{E}}[\mathbf{x}^{\top}\bm{\beta} + \nu + \varepsilon\mid \mathbf{x}]\\
&=\mathbf{x}^{\top}\bm{\beta}+ \hat{\mathbb{E}}[\nu + \varepsilon\mid \mathbf{x}]\\
&=\mathbf{x}^{\top}\bm{\beta}+ \sup \limits_{\theta \in \Theta} \mathbb{E}_{\theta}[\nu + \varepsilon\mid \mathbf{x}]\\
&=\mathbf{x}^{\top}\bm{\beta}+ \sup \limits_{\theta \in \Theta} \mathbb{E}_{\theta}[\nu + \varepsilon]\\
&=\mathbf{x}^{\top}\bm{\beta}+ \sup \limits_{\theta \in \Theta} \{\mathbb{E}_{\theta}[\nu] + \mathbb{E}_{\theta}[\varepsilon]\}\\
&=\mathbf{x}^{\top}\bm{\beta}+ \sup \limits_{\theta \in \Theta} \mathbb{E}_{\theta}[\nu]\\
&=\mathbf{x}^{\top}\bm{\beta}+ \overline{\mu}.
\end{align*}
Similarly, we have
$$
-\hat{\mathbb{E}}[-y\mid \mathbf{x}]=\inf \limits_{\theta \in \Theta} \mathbb{E}_{\theta}[y\mid \mathbf{x}]=\mathbf{x}^{\top}\bm{\beta} + \underline{\mu}.
$$
This completes the proof.
\end{proof}

\begin{proof}[Proof of Lemma \ref{lem:explicit}]
The EM algorithm consists of the following steps upon convergence:

(a) the E-step: For $1\leq j\leq T$, find \(\hat{a}_j\) such that
\[
\sum_{i=1+(j-1)n_0}^{j n_0} \bigl( y_i - \hat{a}_j - \mathbf{x}_i^{\top} \hat{\bm{\beta}}_{\mathrm{EM}} \bigr) = 0,
\]
and set \(\hat{\nu}_i = \hat{a}_j\) for $1+(j-1)n_0\leq i\leq j n_0$. This gives
\[
\hat{a}_j = \bar{y}_j - \bar{\mathbf{x}}_j^{\top} \hat{\boldsymbol{\beta}}_{\mathrm{EM}},
\]
where \(\bar{y}_j = \frac{1}{n_0} \sum_{i=1+(j-1)n_0}^{j n_0} y_i\) and \(\bar{\mathbf{x}}_j = \frac{1}{n_0} \sum_{i=1+(j-1)n_0}^{j n_0} \mathbf{x}_i\).

~\\

(b) the M-step: The estimator \(\hat{\bm{\beta}}_{\mathrm{EM}}\) satisfies
$$
\sum_{i=1}^{n} \bigl( y_i - \hat{\nu}_i - \mathbf{x}_i^{\top} \hat{\bm{\beta}}_{\mathrm{EM}} \bigr) \mathbf{x}_i = 0.
$$
Substituting  $\hat{\nu}_{i}=\bar{y}_j - \bar{\mathbf{x}}_j^{\top} \hat{\boldsymbol{\beta}}_{\mathrm{EM}}$ into the M-step equation yields
\[
\sum_{j=1}^{T} \sum_{i=1+(j-1)n_0}^{j n_0} \bigl( y_i - \bar{y}_j + \bar{\mathbf{x}}_j^{\top} \hat{\boldsymbol{\beta}}_{\mathrm{EM}} - \mathbf{x}_i^{\top} \hat{\boldsymbol{\beta}}_{\mathrm{EM}} \bigr) \mathbf{x}_i = 0.
\]
Rearranging terms, we obtain
\[
\sum_{j=1}^{T} \sum_{i=1+(j-1)n_0}^{j n_0} (y_i - \bar{y}_j) \mathbf{x}_i = \sum_{j=1}^{T} \sum_{i=1+(j-1)n_0}^{j n_0} \mathbf{x}_i \bigl( \mathbf{x}_i  - \bar{\mathbf{x}}_j \bigr)^{\top} \hat{\bm{\beta}}_{\mathrm{EM}} .
\]
Note that \(\sum_{i=1+(j-1)n_0}^{j n_0} (y_i - \bar{y}_j) = 0\) and \(\sum_{i=1+(j-1)n_0}^{j n_0} (\mathbf{x}_i - \bar{\mathbf{x}}_j) = \mathbf{0}\), we have
\[
\sum_{j=1}^{T} \sum_{i=1+(j-1)n_0}^{j n_0} (\mathbf{x}_i - \bar{\mathbf{x}}_j)(y_i - \bar{y}_j) = \Bigl[ \sum_{j=1}^{T} \sum_{i=1+(j-1)n_0}^{j n_0} (\mathbf{x}_i - \bar{\mathbf{x}}_j)(\mathbf{x}_i - \bar{\mathbf{x}}_j)^{\top} \Bigr] \hat{\bm{\beta}}_{\mathrm{EM}}.
\]
Assume that the matrix \(\sum_{j=1}^{T} \sum_{i=1+(j-1)n_0}^{j n_0} (\mathbf{x}_i - \bar{\mathbf{x}}_j)(\mathbf{x}_i - \bar{\mathbf{x}}_j)^{\top}\) is invertible, we have
\[
\hat{\bm{\beta}}_{\mathrm{EM}} = \Bigl[ \sum_{j=1}^{T} \sum_{i=1+(j-1)n_0}^{j n_0} (\mathbf{x}_i - \bar{\mathbf{x}}_j)(\mathbf{x}_i - \bar{\mathbf{x}}_j)^{\top} \Bigr]^{-1} \sum_{j=1}^{T} \sum_{i=1+(j-1)n_0}^{j n_0} (\mathbf{x}_i - \bar{\mathbf{x}}_j)(y_i - \bar{y}_j).
\]
Due to
$$
y_i - \bar{y}_j = (\mathbf{x}_i - \bar{\mathbf{x}}_j)^{\top} \bm{\beta} + (\varepsilon_i - \bar{\varepsilon}_j),
$$
where \(\bar{\varepsilon}_j = \frac{1}{n_0} \sum_{i=1+(j-1)n_0}^{j n_0} \varepsilon_i\). We have
\[
\hat{\bm{\beta}}_{\mathrm{EM}} = \bm{\beta} + \Bigl[ \sum_{j=1}^{T} \sum_{i=1+(j-1)n_0}^{j n_0} (\mathbf{x}_i - \bar{\mathbf{x}}_j)(\mathbf{x}_i - \bar{\mathbf{x}}_j)^{\top} \Bigr]^{-1} \sum_{j=1}^{T} \sum_{i=1+(j-1)n_0}^{j n_0} (\mathbf{x}_i - \bar{\mathbf{x}}_j)(\varepsilon_i - \bar{\varepsilon}_j).
\]
Note that \(\sum_{i=1+(j-1)n_0}^{j n_0} (\mathbf{x}_i - \bar{\mathbf{x}}_j) = \mathbf{0}\), then
\[
\hat{\bm{\beta}}_{\mathrm{EM}} = \bm{\beta} + \Bigl[ \sum_{j=1}^{T} \sum_{i=1+(j-1)n_0}^{j n_0} (\mathbf{x}_i - \bar{\mathbf{x}}_j)(\mathbf{x}_i - \bar{\mathbf{x}}_j)^{\top} \Bigr]^{-1} \sum_{j=1}^{T} \sum_{i=1+(j-1)n_0}^{j n_0} (\mathbf{x}_i - \bar{\mathbf{x}}_j)\varepsilon_i.
\]
This completes the proof.
\end{proof}

\begin{proof}[Proof of Theorem \ref{thm:unbiasednessLEM}]
Let $\mathbf{X}$ be the $n\times p$ covariate matrix. According to Lemma \ref{lem:explicit}, we have
\[
\hat{\bm{\beta}}_{\mathrm{EM}} = \bm{\beta} + \Bigl[ \sum_{j=1}^{T} \sum_{i=1+(j-1)n_0}^{j n_0} (\mathbf{x}_i - \bar{\mathbf{x}}_j)(\mathbf{x}_i - \bar{\mathbf{x}}_j)^{\top} \Bigr]^{-1} \sum_{j=1}^{T} \sum_{i=1+(j-1)n_0}^{j n_0} (\mathbf{x}_i - \bar{\mathbf{x}}_j)\varepsilon_i.
\]
By Assumption \ref{assum:mindependent}, \(\varepsilon_i\) is independent of \(\mathbf{X}\), hence \(\mathbb{E}[\varepsilon_i \mid \mathbf{X}] =\mathbb{E}[\varepsilon_i] =0\). Therefore,
\[
\mathbb{E}\Bigl[ \sum_{j=1}^{T} \sum_{i=1+(j-1)n_0}^{j n_0} (\mathbf{x}_i - \bar{\mathbf{x}}_j)\varepsilon_i \;\Big|\; \mathbf{X} \Bigr] = 0.
\]
Consequently,
\[
\mathbb{E}\bigl[ \hat{\bm{\beta}}_{\mathrm{EM}} \mid \mathbf{X} \bigr] = \bm{\beta}.
\]
Taking unconditional expectation yields $\mathbb{E}\bigl[ \hat{\bm{\beta}}_{\mathrm{EM}} \bigr] = \bm{\beta}$. Thus, \(\hat{\bm{\beta}}_{\mathrm{EM}}\) is an unbiased estimator for \(\bm{\beta}\). This completes the proof.
\end{proof}

\begin{proof}[Proof of Theorem \ref{thm:consistencyLEM fixed}]
According to Lemma \ref{lem:explicit}, we have
\begin{align*}
\hat{\bm{\beta}}_{\mathrm{EM}} &= \bm{\beta} + \Bigl[ \sum_{j=1}^{T} \sum_{i=1+(j-1)n_0}^{j n_0} (\mathbf{x}_i - \bar{\mathbf{x}}_j)(\mathbf{x}_i - \bar{\mathbf{x}}_j)^{\top} \Bigr]^{-1} \sum_{j=1}^{T} \sum_{i=1+(j-1)n_0}^{j n_0} (\mathbf{x}_i - \bar{\mathbf{x}}_j)\varepsilon_i\\
&= \bm{\beta} + \left( \frac{1}{T}\sum_{j=1}^{T} \mathbf{S}_j \right)^{-1} \frac{1}{T} \sum_{j=1}^{T}\mathbf{t}_j,
\end{align*}
where
$$
\mathbf{S}_j=\sum_{i=1+(j-1)n_0}^{j n_0} (\mathbf{x}_i - \bar{\mathbf{x}}_j)(\mathbf{x}_i - \bar{\mathbf{x}}_j)^{\top}, \quad
\mathbf{t}_j=\sum_{i=1+(j-1)n_0}^{j n_0} (\mathbf{x}_i - \bar{\mathbf{x}}_j)\varepsilon_i.
$$
First, since the windows are disjoint and $\{\mathbf{x}_i\}_{i=1}^{n}$ are i.i.d. with finite second moment, $\{\mathbf{S}_j\}_{j=1}^{T}$ are i.i.d. with

\[
\mathbb{E}[\mathbf{S}_j] = \mathbb{E}\left[ \sum_{i=1+(j-1)n_0}^{j n_0} (\mathbf{x}_i - \bar{\mathbf{x}}_j)(\mathbf{x}_i - \bar{\mathbf{x}}_j)^{\top} \right] = (n_0 - 1) \bm{\Sigma}_{\mathbf{xx}}.
\]
By the Law of Large Numbers for i.i.d. random matrices,
\[
\frac{1}{T}\sum_{j=1}^{T} \mathbf{S}_j \xrightarrow{P} (n_0 - 1)\bm{\Sigma}_{\mathbf{xx}}.
\]
The positive definiteness of $\bm{\Sigma}_{\mathbf{xx}}$ and continuity of the matrix inverse then yield
\begin{equation}\label{eq:bc1}
\left( \frac{1}{T}\sum_{j=1}^{T} \mathbf{S}_j\right)^{-1} \xrightarrow{P} [(n_0 - 1)\bm{\Sigma}_{\mathbf{xx}}]^{-1}.
\end{equation}

Second, recall that $\tau_j^2$ is the error variance within the $j$-th window. Note that $\{(\mathbf{x}_i, \varepsilon_i)\}_{i=1}^{n}$ are independent across $i$, and $\varepsilon_i$ is independent of $\mathbf{x}_i$ for each $i$.
It follows that $\{\mathbf{t}_j\}_{j=1}^{T}$ are independent, satisfying
\[
\mathbb{E}[\mathbf{t}_j] = \mathbb{E}\left[\sum_{i=1+(j-1)n_0}^{j n_0} (\mathbf{x}_i - \bar{\mathbf{x}}_j)\varepsilon_i\right]= \sum_{i=1+(j-1)n_0}^{j n_0}\mathbb{E}[\mathbf{x}_i - \bar{\mathbf{x}}_j]\mathbb{E}[\varepsilon_i] = \mathbf{0},
\]
and
\begin{align*}
\mathrm{Var}[\mathbf{t}_j]&=\mathbb{E}[\mathbf{t}_j\mathbf{t}_{j}^{\top}]\\
&= \mathbb{E}\left[ \sum_{i=1+(j-1)n_0}^{j n_0} \sum_{l=1+(j-1)n_0}^{j n_0} (\mathbf{x}_i - \bar{\mathbf{x}}_j)(\mathbf{x}_l - \bar{\mathbf{x}}_j)^\top \varepsilon_i \varepsilon_l \right] \\
&= \tau_{j}^2 \mathbb{E}\left[ \sum_{i=1+(j-1)n_0}^{j n_0} (\mathbf{x}_i - \bar{\mathbf{x}}_j)(\mathbf{x}_i - \bar{\mathbf{x}}_j)^\top \right] \\
&= (n_0 - 1)\tau_{j}^2 \bm{\Sigma}_{\mathbf{xx}},
\end{align*}
where the third equality follows from the independence of $\varepsilon_i$ and $\mathbf{x}_i$, along with the fact that $\mathbb{E}[\varepsilon_i\varepsilon_l] = 0$ for $i \neq l$.

Since $\tau_{j}^2$ is uniformly bounded by $\overline{\sigma}^2$, we have $\sup_{j} \|\mathrm{Var}[\mathbf{t}_j]\| < \infty$. By the Weak Law of Large Numbers for independent random vectors with  uniformly bounded variances,
\begin{equation}\label{eq:bc2}
\frac{1}{T} \sum_{j=1}^{T}\mathbf{t}_j \xrightarrow{P} \mathbf{0}.
\end{equation}
Combining \eqref{eq:bc1} and \eqref{eq:bc2}, we have
\[
\hat{\bm{\beta}}_{\mathrm{EM}} - \bm{\beta}=\left( \frac{1}{T}\sum_{j=1}^{T} \mathbf{S}_j \right)^{-1} \frac{1}{T} \sum_{j=1}^{T}\mathbf{t}_j \xrightarrow{P} \mathbf{0},
\]
which is equivalent to $\hat{\bm{\beta}}_{\mathrm{EM}} \xrightarrow{P} \bm{\beta}$.
This completes the proof.
\end{proof}

\begin{proof}[Proof of Theorem \ref{thm:ANLEM fixed}]
From the proof of Theorem~\ref{thm:consistencyLEM fixed}, we have
$$
\sqrt{T}(\hat{\bm{\beta}}_{\mathrm{EM}} - \bm{\beta})=\left( \frac{1}{T}\sum_{j=1}^{T} \mathbf{S}_j \right)^{-1} \frac{1}{\sqrt{T}} \sum_{j=1}^{T}\mathbf{t}_j,
$$
where
$$
\mathbf{S}_j=\sum_{i=1+(j-1)n_0}^{j n_0} (\mathbf{x}_i - \bar{\mathbf{x}}_j)(\mathbf{x}_i - \bar{\mathbf{x}}_j)^{\top}, \quad
\mathbf{t}_j=\sum_{i=1+(j-1)n_0}^{j n_0} (\mathbf{x}_i - \bar{\mathbf{x}}_j)\varepsilon_i.
$$
We have already shown in the proof of Theorem~\ref{thm:consistencyLEM fixed} that
\begin{equation}\label{eq:ban1}
\left( \frac{1}{T}\sum_{j=1}^{T} \mathbf{S}_j\right)^{-1} \xrightarrow{P} [(n_0 - 1)\bm{\Sigma}_{\mathbf{xx}}]^{-1}.
\end{equation}
Moreover, $\{\mathbf{t}_j\}_{j=1}^{T}$ are independent with $\mathbb{E}[\mathbf{t}_j] = \mathbf{0}$ and
$$
\mathrm{Var}[\mathbf{t}_j]=(n_0 - 1)\tau_{j}^2 \bm{\Sigma}_{\mathbf{xx}}.
$$
Let
$$
\mathbf{B}_T^2 = \sum_{j=1}^T \mathrm{Var}[\mathbf{t}_j]=(n_0-1)\bm{\Sigma}_{\mathbf{xx}} \sum_{j=1}^T \tau_j^2,
$$
since $\tau_{j}^2\in [\underline{\sigma}^2,\overline{\sigma}^2]$ and $\bm{\Sigma}_{\mathbf{xx}}$ is positive definite, we have
\[
\lambda_{\min}(\mathbf{B}_T^2) = (n_0-1)\lambda_{\min}(\bm{\Sigma}_{\mathbf{xx}}) \sum_{j=1}^T \tau_j^2 \ge (n_0-1)\lambda_{\min}(\bm{\Sigma}_{\mathbf{xx}}) T\underline{\sigma}^2,
\]
where $\lambda_{\min}(\bm{\Sigma}_{\mathbf{xx}}) > 0$ is the minimal eigenvalue of $\bm{\Sigma}_{\mathbf{xx}}$.

Conditioned on $\{\mathbf{x}_i\}_{i=1}^{n}$, $\mathbf{t}_j$ follows a multivariate normal distribution. By the property of Gaussian moments, there exists a constant $C_p>0$ (depending only on $p$) such that
\[
\mathbb{E}\!\left[\|\mathbf{t}_j\|^3 \;\big|\; \{\mathbf{x}_i\}_{i=1}^{n}\right] \le C_p \tau_j^3 \left(\text{tr}(\mathbf{S}_j)\right)^{3/2} = C_p \tau_j^3 \left(\sum_{i=1+(j-1)n_0}^{j n_0} \|\mathbf{x}_i-\bar{\mathbf{x}}_j\|^2\right)^{3/2}.
\]
Taking unconditional expectation and using $\tau_j \le \overline{\sigma}$ and $\mathbb{E}[\|\mathbf{x}\|^3] < \infty$, we obtain
\[
\mathbb{E}[\|\mathbf{t}_j\|^3]\le C_p \overline{\sigma}^3 \mathbb{E}\!\left[\left(\sum_{i=1+(j-1)n_0}^{j n_0} \|\mathbf{x}_i-\bar{\mathbf{x}}_j\|^2\right)^{3/2}\right] \le C_0 \overline{\sigma}^3,
\]
where $C_0$ is a constant independent of $j$. Thus,
\[
\frac{1}{[\lambda_{\min}(\mathbf{B}_T^2)]^{3/2}} \sum_{j=1}^T \mathbb{E}[\|\mathbf{t}_j\|^3]
\le \frac{T C_0 \overline{\sigma}^3}{\left[(n_0-1)\lambda_{\min}(\bm{\Sigma}_{\mathbf{xx}}) T\underline{\sigma}^2\right]^{3/2}}
= \frac{C_0 \overline{\sigma}^3}{(n_0-1)^{3/2} \lambda_{\min}(\bm{\Sigma}_{\mathbf{xx}})^{3/2} \underline{\sigma}^3 \sqrt{T}},
\]
which implies that
\[
\lim_{T\to\infty} \frac{1}{[\lambda_{\min}(\mathbf{B}_T^2)]^{3/2}} \sum_{j=1}^T \mathbb{E}[\|\mathbf{t}_j\|^3] = 0.
\]
It follows directly from the Multivariate Lyapunov Central Limit Theorem that
\begin{equation}\label{eq:ban2}
 \frac{1}{\sqrt{T}} \sum_{j=1}^{T}\mathbf{t}_j \xrightarrow{d} \mathcal{N}(\mathbf{0}, \tilde{\sigma}_{T}^2 (n_0 - 1) \bm{\Sigma}_{\mathbf{xx}}),
\end{equation}
where $\tilde{\sigma}_{T}^2=\frac{1}{T}\sum_{j=1}^{T} \tau_{j}^2\in[\underline{\sigma}^2,\overline{\sigma}^2]$.

Applying Slutsky's theorem to \eqref{eq:ban1} and \eqref{eq:ban2}, we obtain
\[
\sqrt{T}(\hat{\bm{\beta}}_{\mathrm{EM}} - \bm{\beta}) \xrightarrow{d} \mathcal{N}(\mathbf{0}, \frac{\tilde{\sigma}_{T}^2}{n_0 - 1} \bm{\Sigma}_{\mathbf{xx}}^{-1}).
\]
Thus,
\[
\sqrt{n}(\hat{\bm{\beta}}_{\mathrm{EM}} - \bm{\beta}) \xrightarrow{d} \mathcal{N}(\mathbf{0}, \frac{n_{0}}{n_0 - 1} \tilde{\sigma}_{T}^2\bm{\Sigma}_{\mathbf{xx}}^{-1}).
\]
This completes the proof.
\end{proof}

\begin{proof}[Proof of Theorem \ref{thm:consistency_bounds}]
We first prove the consistency of $\hat{\overline{\mu}}$. Recall that for the $l$th block $B_{l}=\{(\mathbf{x}_{i},y_{i})\}_{i=l}^{l+w-1}$, we compute $a_l$ by solving
\[
\sum_{i=l}^{l+w-1} \bigl\{ y_i - (\mathbf{x}_i^{\top}\hat{\bm{\beta}}_{\mathrm{EM}} + a_l) \bigr\} = 0.
\]
This yields
\[
a_l = \frac{1}{w} \sum_{i=l}^{l+w-1} (y_i - \mathbf{x}_i^{\top}\hat{\bm{\beta}}_{\mathrm{EM}}).
\]
Substituting $y_i = \mathbf{x}_i^{\top}\bm{\beta}+ \nu_i + \varepsilon_i$ gives
\[
a_l = \underbrace{\frac{1}{w} \sum_{i=l}^{l+w-1} \nu_i}_{\textstyle :=\, \mu_l}
      + \underbrace{\frac{1}{w} \sum_{i=l}^{l+w-1} \mathbf{x}_i^{\top}(\bm{\beta} -\hat{\bm{\beta}}_{\mathrm{EM}})}_{\textstyle :=\, \delta_{1,l}}
      + \underbrace{\frac{1}{w} \sum_{i=l}^{l+w-1} \varepsilon_i}_{\textstyle :=\, \delta_{2,l}}.
\]
Here, $\mu_l = \mu(B_l)$ is precisely the expected average intercept for block $B_l$ as given in Definition~\ref{def:wbounds}, and by construction $\overline{\mu}_w = \max_{1\leq l\leq L} \mu_l$.

First, by Theorem \ref{thm:consistencyLEM fixed} we have $\hat{\bm{\beta}}_{\mathrm{EM}} \xrightarrow{P} \bm{\beta}$,
and Assumption \ref{assum:x} ($\mathbb{E}[\|\mathbf{x}\|^3] < \infty$) implies that $\mathbb{E}[\|\mathbf{x}\|] < \infty$, thus
\[
\max_{1 \leq l \leq L} |\delta_{1,l}| \leq \|\hat{\bm{\beta}}_{\mathrm{EM}} - \bm{\beta}\| \cdot \max_{1 \leq l \leq L}\frac{1}{w}\sum_{i=l}^{l+w-1}\|\mathbf{x}_i\| = o_p(1)\cdot O_p(1) = o_p(1).
\]

Second, since $\{\varepsilon_i\}_{i=1}^{n}$ are independent normal random variables with zero mean and variances bounded by $\overline{\sigma}^2$, they are sub-Gaussian. Hence, there exists a constant $C > 0$ (e.g., $C = 2\overline{\sigma}^2$) such that for each block $B_{l}$ and any $t > 0$,

\[
P\left( |\delta_{2,l}| > t \right) \le 2 \exp\left( -\frac{w t^2}{C} \right).
\]
Applying a union bound over all $L = n - w + 1$ blocks yields
\[
P\left( \max_{1 \le l \le L} |\delta_{2,l}| > t \right)\le \sum_{l=1}^{L} P\left( |\delta_{2,l}| > t \right)\le 2L \exp\left( -\frac{w t^2}{C} \right)< 2n \exp\left( -\frac{w t^2}{C} \right).
\]
By Assumption \ref{assum:w}, $w / \log n \to \infty$. Thus, for any fixed $t > 0$,
\[
n \exp\left( -\frac{w t^2}{C} \right) = \exp\left( \log n - \frac{w t^2}{C} \right) \to 0,
\]
which implies
\[
\max_{1 \le l \le L} |\delta_{2,l}| = o_p(1).
\]
Consequently, $\max_{1 \leq l \leq L} |a_l - \mu_l|=\max_{1 \leq l \leq L} |\delta_{1,l} + \delta_{2,l}| = o_p(1)$.

Since $\hat{\overline{\mu}}=\max_{1\leq l\leq L}a_{l}$ and $\overline{\mu}_w = \max_{1 \leq l \leq L} \mu_l$ , it follows immediately that

\[
|\hat{\overline{\mu}} - \overline{\mu}_w| = \left| \max_{1 \leq l \leq L} a_l - \max_{1 \leq l \leq L} \mu_l \right| \leq \max_{1 \leq l \leq L} |a_l - \mu_l| = o_p(1),
\]
and thus
$$
\hat{\overline{\mu}} \xrightarrow{P} \overline{\mu}_w.
$$
Similarly, we can prove that
$$
\hat{\underline{\mu}} \xrightarrow{P} \underline{\mu}_w.
$$
This completes the proof.
\end{proof}

\bibliographystyle{plainnat}
\bibliography{GLinear_reference}

@article{jin2021optimal,
  title={Optimal unbiased estimation for maximal distribution},
  author={Jin, Hanqing and Peng, Shige},
  journal={Probability, Uncertainty and Quantitative Risk},
  volume={6},
  number={3},
  pages={189--198},
  year={2021},
  publisher={Probability, Uncertainty and Quantitative Risk}
}

@article{lin2016k,
  title={k-sample upper expectation linear regression---{Modeling}, identifiability, estimation and prediction},
  author={Lin, Lu and Shi, Yufeng and Wang, Xin and Yang, Shuzhen},
  journal={Journal of Statistical Planning and Inference},
  volume={170},
  pages={15--26},
  year={2016},
  publisher={Elsevier}
}

@article{lin2017upper,
  title={Upper expectation parametric regression},
  author={Lin, Lu and Dong, Ping and Song, Yunquan and Zhu, Lixing},
  journal={Statistica Sinica},
  pages={1265--1280},
  year={2017},
  publisher={JSTOR}
}

@conference{PengICM2010,
    author ={Peng, Shige} ,
    booktitle ={ICM 2010} ,
    title ={Backward Stochastic Differential Equation, nonlinear expectations and their applications},
    year = 2010,
    address = {Hyderabad, India}
}

@article{peng2004filtration,
  title={Filtration consistent nonlinear expectations and evaluations of contingent claims},
  author={Peng, Shige},
  journal={Acta Mathematicae Applicatae Sinica, English Series},
  volume={20},
  pages={191--214},
  year={2004},
  publisher={Springer}
}

@article{peng2005nonlinear,
  title={Nonlinear expectations and nonlinear {Markov} chains},
  author={Peng, Shige},
  journal={Chinese Annals of Mathematics},
  volume={26},
  number={02},
  pages={159--184},
  year={2005},
  publisher={World Scientific}
}

@book{peng2019nonlinear,
  title={Nonlinear expectations and stochastic calculus under uncertainty: with robust CLT and G-Brownian motion},
  author={Peng, Shige},
  volume={95},
  year={2019},
  publisher={Springer Nature}
}

@article{peng2020hypothesis,
  title={A hypothesis-testing perspective on the {G-normal} distribution theory},
  author={Peng, Shige and Zhou, Quan},
  journal={Statistics and Probability Letters},
  volume={156},
  pages={108623},
  year={2020},
  publisher={Elsevier}
}

@article{yang2023linear,
  title={Linear regression under model uncertainty},
  author={Shuzhen Yang and Jianfeng Yao},
  journal={Probability, Uncertainty and Quantitative Risk},
  volume = {8},
  number = {4},
  pages = {523-546},
  year = {2023}
}

@article{li2026solution,
  title={A solution to the p-hacking problem with a general robust significance test under variance uncertainty},
  author={Li, Xifeng and Yang, Shuzhen and Yao, Jianfeng},
  journal={Statistical Papers},
  volume={67},
  number={1},
  pages={14},
  year={2026},
  publisher={Springer}

}

@article{liang2015assessing,
  title={Assessing {B}eijing's {PM}2.5 pollution: severity, weather impact, {APEC} and winter heating},
  author={Liang, Xuan and Zou, Tao and Guo, Bin and Li, Shuo and Zhang, Haozhe and Zhang, Shuyi and Huang, Hui and Chen, Song Xi},
  journal={Proceedings of the Royal Society A: Mathematical, Physical and Engineering Sciences},
  volume={471},
  number={2182},
  year={2015},
  publisher={The Royal Society}
}

@article{hu2023arbitrage,
  title={Arbitrage pricing with heterogeneous spatial effects and heteroscedastic disturbances},
  author={Hu, Jianhua and Ding, Hao and Liu, Xiaoqian},
  journal={Journal of Financial Econometrics},
  volume={21},
  number={4},
  pages={1169--1195},
  year={2023},
  publisher={Oxford University Press}
}

@article{eberhardt2011econometrics,
  title={Econometrics for grumblers: a new look at the literature on cross-country growth empirics},
  author={Eberhardt, Markus and Teal, Francis},
  journal={Journal of economic Surveys},
  volume={25},
  number={1},
  pages={109--155},
  year={2011},
  publisher={Wiley Online Library}
}

@article{austin2017intermediate,
  title={Intermediate and advanced topics in multilevel logistic regression analysis},
  author={Austin, Peter C and Merlo, Juan},
  journal={Statistics in medicine},
  volume={36},
  number={20},
  pages={3257--3277},
  year={2017},
  publisher={Wiley Online Library}
}

@article{pouliot2016robust,
  title={Robust tests for change in intercept and slope in linear regression models with application to manager performance in the mutual fund industry},
  author={Pouliot, William},
  journal={Economic Modelling},
  volume={58},
  pages={523--534},
  year={2016},
  publisher={Elsevier}
}

@article{fan2024environment,
  title={Environment invariant linear least squares},
  author={Fan, Jianqing and Fang, Cong and Gu, Yihong and Zhang, Tong},
  journal={The Annals of Statistics},
  volume={52},
  number={5},
  pages={2268--2292},
  year={2024},
  publisher={Institute of Mathematical Statistics}
}

@article{yan2024statistical,
  title={Statistical inference for four-regime segmented regression models},
  author={Yan, Han and Chen, Song Xi},
  journal={The Annals of Statistics},
  volume={52},
  number={6},
  pages={2668--2691},
  year={2024},
  publisher={Institute of Mathematical Statistics}
}

@article{dempster1977maximum,
  title={Maximum likelihood from incomplete data via the {EM} algorithm},
  author={Dempster, A. P. and Laird, N. M. and Rubin, D. B.},
  journal={Journal of the Royal Statistical Society: Series B (methodological)},
  volume={39},
  number={1},
  pages={1--22},
  year={1977}
}

@misc{EPA_ECATT,
  author       = {{U.S. Environmental Protection Agency}},
  title        = {Air Monitoring Station {HAP} Data Rules and Calculations},
  year         = {2025},
  howpublished = {Available at: \url{https://echo.epa.gov/help/ecatt/air-monitoring-station-data-calculations}}
}

\end{document}